\documentclass[11pt]{amsart}
\usepackage[english]{babel}
\usepackage[T1]{fontenc}
\usepackage[ansinew]{inputenc}
\usepackage{amsmath}
\usepackage{amsfonts}
\usepackage{ mathrsfs }
\usepackage{amssymb}
\usepackage{textcomp}
\usepackage{enumitem}
\usepackage[hidelinks]{hyperref}
\usepackage[arrow, matrix, curve]{xy}
\usepackage{comment}
\usepackage{color}
\usepackage{stmaryrd}

\numberwithin{paragraph}{section}
\numberwithin{equation}{section}

\newtheorem{satz}{Theorem}[section]

\newtheorem{lem}[satz]{Lemma}

\newtheorem{prop}[satz]{Proposition}

\newtheorem{kor}[satz]{Corollary}
\theoremstyle{definition}
\newtheorem{defn}[satz]{Definition}
\newtheorem{bem}[satz]{Remark}

\newtheorem{theointro}{Theorem}

\newtheorem*{ack}{Acknowledgements}

\newcommand{\Z}{\mathbb{Z}}

\newcommand{\Q}{\mathbb{Q}}
\newcommand{\R}{\mathbb{R}}
\newcommand{\C}{\mathbb{C}}

\newcommand{\vedge}{\land}
\newcommand{\del}{\partial}
\newcommand{\Xan}{X^{\an}}

\newcommand{\Ktilde}{\widetilde{K}}

	\DeclareMathOperator{\an}{an}
	\DeclareMathOperator{\PD}{PD}

	\DeclareMathOperator{\Hom}{Hom}

	\DeclareMathOperator{\im}{im}

	\DeclareMathOperator{\sing}{sing}
	\DeclareMathOperator{\supp}{supp}

	\DeclareMathOperator{\divisor}{div}
	
	\DeclareMathOperator{\Pic}{Pic}

	\DeclareMathOperator{\red}{red}

	\DeclareMathOperator{\CH}{CH}
	\DeclareMathOperator{\Aff}{\mathcal{K}}
	\DeclareMathOperator{\AT}{L}
	\DeclareMathOperator{\ddc}{dd^c}
	\DeclareMathOperator{\HH}{H}
	\DeclareMathOperator{\cl}{cl}
	\DeclareMathOperator{\Div}{Div}

\newcommand{\HHc}{\HH_c}

\DeclareMathOperator{\AS}{\mathcal{A}}

\DeclareMathOperator{\FS}{\mathcal{F}}

\DeclareMathOperator{\HS}{\mathcal{H}}

\DeclareMathOperator{\KS}{\mathcal{K}}

\DeclareMathOperator{\OS}{\mathcal{O}}

\DeclareMathOperator{\Scal}{\mathcal{S}} 

\DeclareMathOperator{\US}{\mathcal{U}}
\DeclareMathOperator{\VS}{\mathcal{V}}

\DeclareMathOperator{\XS}{\mathcal{X}}

\def\quotient#1#2{\raise0.75ex\hbox{$\,#1$}\big/\lower0.75ex\hbox{$#2\,$}}

\title{Tropical Hodge numbers of non-archimedean curves}

 \author[P.~Jell]{Philipp Jell}
 \address{P. Jell,  
Georgia Institute of Technology,
686 Cherry Street,
Atlanta, GA 30332-0160}
\email{philipp.jell@math.gatech.edu}

\thanks{The author was supported by the collaborative research center CRC 1085 "Higher Invariants" funded by the Deutsche Forschungsgemeinschaft.
 }

\begin{document}
\begin{abstract}
We study the tropical Dolbeault cohomology of non-archimedean curves
as defined by Chambert--Loir and Ducros. 
We give a precise condition for when this cohomology satisfies Poincar\'e duality.
The condition is always satisfied when the residue field of the non-archimedean base field is the algebraic 
closure of a finite field.
We also show that for curves over fields with residue field $\C$, the tropical $(1,1)$-Dolbeault cohomology can
be infinite dimensional.  
Our main new ingredient is an exponential type sequence that relates tropical Dolbeault 
cohomology to the cohomology of the sheaf of harmonic functions.
As an application of our Poincar\'e duality result, 
we calculate the dimensions of the tropical Dolbeault cohomology, called tropical Hodge numbers, for (open subsets 
of) curves.

\bigskip

\end{abstract}

\maketitle


\section{Introduction}

Let $K$ be a field that is complete with respect to a non-archimedean absolute value. 
For a $K$-analytic space $X$ in the sense of V. Berkovich \cite{BerkovichSpectral}
A. Chambert-Loir and A. Ducros introduced a double complex 
$(\AS^{\bullet, \bullet}, d', d'')$ 
of sheaves of smooth real-valued differential forms on $X$ \cite{CLD}. 
These bigraded differential forms are analogues of smooth differential forms on complex manifolds. 
We consider these forms when the analytic space $\Xan$ is the associated analytic space of a variety $X$ over $K$.
We always assume that $K$ is algebraically closed and that its absolute value is non-trivial. 

A natural question that arose in this context concerns the behavior of the associated Dolbeault cohomology 
\begin{align*}
\HH^{p,q}(\Xan) := \frac {\ker (d'' \colon \AS^{p,q}(\Xan) \to \AS^{p,q+1}(\Xan))}{\im (d'' \colon \AS^{p,q-1}(\Xan) \to \AS^{p,q}(\Xan))}.
\end{align*} 
and the associated tropical Hodge numbers $h^{p,q}(\Xan) := \dim_\R \HH^{p,q}(\Xan)$. 
Concretely, one may ask whether, for smooth proper varieties $X$ of dimension $n$, 
this cohomology satisfies a version of Poincar\'e duality, 
i.e.~if the canonical pairing $\HH^{p,q}(\Xan) \times \HH^{n-p,n-q}(\Xan) \to \R$ is a perfect pairing for all $p$ and $q$. 
This implies the symmetry $h^{p,q}(\Xan) = h^{n-p,n-q}(\Xan)$ for the tropical Hodge numbers. 
We give a precise condition when Poincar\'e duality holds for curves. 

\begin{theointro} [Theorem \ref{PD}, Remark \ref{remark semistable}] \label{ThmA}
Let $X$ be a smooth projective curve over $K$.
Then $X$ satisfies Poincar\'e duality
if and only if 
there exists a strictly semistable model of $X$ over the valuation ring of $K$ such
that for every irreducible component $C$ of the special fiber the abelian group $\Pic^0(C)$ is torsion. 
\end{theointro}

Let us state some consequences and remarks.
\begin{itemize}
\item
Theorem \ref{ThmA} implies that all smooth projective curves over $K$ 
satisfy Poincar\'e duality if the residue field of $K$ is the algebraic closure of a finite field (see Corollary \ref{corollary finite field}). 
\item
Theorem \ref{ThmA} enables us to calculate the tropical Hodge numbers $h^{p,q}(\Xan)$ 
(see Theorem \ref{thm dimensions global}). 
\item 
Using Theorem \ref{ThmA} we are able to establish 
the case of curves for a conjecture by Y. Liu \cite[Conjecture 5.2]{Liu2} (see Theorem \ref{proof Liu conjecture}). 
\item
The property that $\Pic^0(C)$ is torsion for all irreducible components of the special fiber holds 
for all strictly semistable models of $X$ if it holds for one strictly semistable model of $X$  (see Remark \ref{remark semistable}).
\end{itemize}

We also show that in general $\HH^{1,1}(\Xan)$ is not finite dimensional. 

\begin{theointro} [Theorem \ref{non finite dimensional}, Remark \ref{remark semistable}] \label{ThmB}
Let $X$ be a smooth projective curve over $K$ and let the residue field of $K$ be $\C$.  
Then $\HH^{1,1}(X)$ is finite dimensional if and only if 
there exists a strictly semistable model of $X$
such that the special fiber has only rational irreducible components.  
\end{theointro}

Again, if the property that all irreducible components of the special fiber are rational holds  
for one strictly semistable model, it holds for all strictly semistable models (see Remark \ref{remark semistable}). 
In particular for $K$ as in Theorem \ref{ThmB}, the real vector space $\HH^{1,1}(\Xan)$ is infinite dimensional 
for all smooth projective curves $X$ of positive genus and good reduction.

\subsection{Relation with previous results}

Now let $K$ again be any algebraically closed complete non-archimedean field 
and let $X$ be a variety over $K$. 
The author showed in \cite[Corollary 4.6]{Jell} that for all $q$ the cohomology $\HH^{0,q}(\Xan)$ agrees with $\HH^q_{\sing}(\Xan, \R)$ and 
is therefore finite dimensional. 
In his thesis he further showed that for proper $X$ the vector spaces $\HH^{1,0}(\Xan)$ and 
$\HH^{\dim(X),0}(\Xan)$ are finite dimensional. 
He did this by constructing injective maps
$\HH^{p,0}(\Xan) \to \HH^{0,p}(\Xan)$ for $p = 1$ and $p = \dim(X)$ \cite[Proposition 3.4.11]{JellThesis}. 
This was shown for all $p$ by Y. Liu when $K$ is a closed subfield of the completed algebraic closure
of a field $K_0$, where $K_0$ is either a field of characteristic $0$ or $K_0 = k((t))$ for a finite field $k$ \cite[Theorem 1.1]{Liu2}

Liu further showed that there are cycle class maps $\CH^q(X) \to \HH^{q,q}(\Xan)$ \cite[Theorem 1.1]{Liu}. 
He additionally showed that for any smooth proper variety $X$ over $\C_p$ the space 
$\HH^{1,1}(\Xan)$ is finite dimensional \cite[Theorem 1.9]{Liu}  
and one dimensional when $X$ is a curve \cite[Theorem 7.3 (2)]{Liu}.
His techniques include proving a weight filtration on the cohomology of the analytic de Rham complex. 

Liu further constructed so called monodromy maps $N \colon \HH^{p,q}(\Xan) \to \HH^{q,p}(\Xan)$ for $p \geq q$ and conjectures 
that these are isomorphisms 
when $X$ is a smooth proper variety over an algebraically closed, non-archimedean field whose residue field 
is the algebraic closure of a finite field \cite[Conjecture 5.2]{Liu2}. 
He showed that the map $N \colon \HH^{1,0}(\Xan) \to \HH^{0,1}(\Xan)$ is an isomorphism 
under the following assumptions: 
Firstly, he assumed that $X$ is the base change of a smooth proper variety $X_0$ over a field $K_0$
that is a finite extension of $\Q_p$ or 
$k (( t))$ for a finite field $k$. 
Further he supposes that $X_0$ admits a proper strictly semistable model over $K_0^\circ$ \cite[Theorem 1.1]{Liu2}.
His techniques use arithmetic geometry, in particular the weight spectral sequence. 
Theorem \ref{ThmA} implies Liu's conjecture for curves in full generality (see Theorem \ref{proof Liu conjecture}).

For Mumford curves V. Wanner and the author showed that Poincar\'e duality holds 
over any complete non-archimedean algebraically closed field $K$ \cite[Theorem 1]{JellWanner}.
The techniques in their work are tropical.
They apply previous work by K. Shaw, J. Smacka and the author on Poincar\'e duality for smooth tropical varieties \cite{JSS}. 
Wanner and the author also obtained a Poincar\'e duality result for certain open subsets of $\Xan$ 
for arbitrary smooth projective curves $X$ \cite[Theorem 2]{JellWanner}.
In Section \ref{sect. global to local} we generalize this result 
to open subsets which form a basis of the topology of $\Xan$ for any smooth curve $X$ 
that satisfies the conditions of Theorem \ref{ThmA}. 
The mentioned basis of the topology consists of so called \emph{strictly simple open subsets} (see Definition \ref{Def. simple}). 
Our proof crucially relies on \cite[Theorem 2]{JellWanner}.

\begin{theointro} [Theorem \ref{Theorem Teilmenge Koho.}] \label{ThmC}
Let $X$ be a smooth projective curve over $K$ that satisfies Poincar\'e duality. 
Let $U$ be a strictly simple open subset of $\Xan$ and denote by $k := \# \del U$ the finite number of boundary points. 
Then we have 
	\begin{align*}
	h^{p,q}(U) = 
	\begin{cases}
	1 &\text{ if } (p,q) = (0,0) \\
	k - 1 &\text{ if } (p,q) = (1,0) \\
	0 &\text{ else}.
	\end{cases}
	\end{align*}
\end{theointro}

 \subsection{Notation and notions in other papers}

Note that in \cite{Jell} and \cite{JellThesis} 
the author considers 
cohomology with respect to $d'$ instead of $d''$. 
Therefore, when translating results to our setting,
one has to flip the bigrading.

\subsection{Techniques and outline of the paper}

We study tropical Dolbeault cohomology by relating it to the cohomology 
of the sheaf $\KS := \ker(d'd'' \colon \AS^{0,0} \to \AS^{1,1})$. 
The key new ingredient that establishes this relation is a short exact sequence,
which might be thought of as a non-archimedean analogue of the exponential sequence
in complex geometry. 
A tropical analogue of this short exact sequence was introduced by G. Mikhalkin and I. Zharkov 
to define 1-forms on tropical curves \cite{MikZharII}.   
We prove a non-archimedean version in Proposition \ref{thm exponential sequence}. 
The terms of the associated long exact sequence include $\HH^{1,0}(\Xan)$, $\HH^{0,1}(\Xan)$, $\HH^{1,1}(\Xan)$
and $\HH^{1}(\Xan, \KS)$.  
This sequence implies that Poincar\'e duality holds for $X$ if and only if $\HH^1(\Xan, \KS)$ is one-dimensional 
and that $\HH^{1,1}(\Xan)$ is finite dimensional if and only if $\HH^{1}(\Xan, \KS)$ is.

Our technique to calculate $\HH^1(\Xan, \KS)$ involves Thuillier's sheaf of harmonic functions $\HS_X$, 
which naturally contains $\Aff$ as a subsheaf. 
Using Thuillier's notion of lisse functions and 1-forms we show that $\HH^1(\Xan, \HS_X)$ is one dimensional. 
By \cite[Lemme 2.3.22]{Thuillier} and \cite[Theorem 5.2.4]{Wanner} 
the quotient $\HS_X / \KS$  is a finite sum of skyscraper sheaves and
vanishes if and only if the conditions of Theorem \ref{ThmA} are satisfied 
(see Definition \ref{defn G(X)} and Proposition \ref{seq harmonic}). 

This enables us to prove Theorem \ref{ThmA} using the above mentioned exact sequence and 
our calculation of the cohomology $\HH^1(\Xan, \HS_X)$.  
For the proof of Theorem \ref{ThmB} we need to use that $\HH^0(\Xan, \HS_X / \KS)$ 
may have infinite dimension when the residue field of $K$ is $\C$. 

To prove Theorem \ref{ThmC} we use Theorem \ref{ThmA}, a Mayer-Vietoris argument and 
a result by Wanner and the author \cite[Theorem 2]{JellWanner}
to show that if $X$ satisfies the conditions of Theorem \ref{ThmA}, 
open subsets of $\Xan$ satisfy a local version of Poincar\'e duality (see Theorem \ref{PD local}). 
While Mayer-Vietoris sequences are normally used to pass from local results to global results, 
we use our global result from Theorem \ref{ThmA} and the fact that we already have local results at
all but finitely many points to deduce local results at all points. 
We then use results on the local structure of non-archimedean curves by M. Baker, S. Payne and J. Rabinoff \cite{BPR2} and 
identification of tropical Dolbeault cohomology with singular cohomology to prove Theorem \ref{ThmC}.

\begin{ack}
The author would like to thank Walter Gubler, Klaus K\"unnemann and Veronika Wanner for reading 
previous drafts of this work and providing useful comments. 
Further he would like to thank Yifeng Liu, Johannes Rau and Kristin Shaw for helpful suggestions. 
Finally the author would like to thank the referee for their very precise report and helpful suggestions, 
which helped a lot in clarifying the presentation of the paper. 
\end{ack}

\section{Preliminaries}

Throughout this paper, $K$ will be an algebraically closed field that is complete with respect
to a non-archimedean absolute value, which is assumed to be non-trivial. 
We denote its valuation ring by $K^\circ$ and its residue field by $\widetilde{K}$. 

Further $X$ will be a curve over $K$, by which we mean an irreducible, reduced, separated, one-dimensional $K$-scheme 
of finite type.

\subsection{Analytification}

To the curve $X$, one can associate an analytic space $\Xan$ over $K$ in the sense of Berkovich \cite{BerkovichSpectral}. 
We will merely concern ourselves with the underlying topological space of $\Xan$, 
which is a paracompact Hausdorff space. 
This topological space is homotopy equivalent to a finite one-dimensional simplicial complex. 
If $X$ is projective, then $\Xan$ is further compact. 

The points in $\Xan$ can be classified into four types \cite[p.~27]{Thuillier} (extending results by Berkovich \cite[p.~19]{BerkovichSpectral}). 
A point $x \in \Xan$ is of type II if the residue field $\widetilde{\mathscr{H}}(x)$ of the completed residue field $\mathscr{H}(x)$ at $x$ is an extension 
of $\widetilde{K}$ of transcendence degree $1$. 
We denote by $C_x$ the smooth projective $\widetilde{K}$-curve with function field $\widetilde{\mathscr{H}}(x)$.

\subsection{Thuillier's sheaf of harmonic functions}

We recall Thuillier's notion of lisse functions, $1$-forms and harmonic functions 
on $\Xan$ for a smooth curve $X$ and their properties.

\begin{defn}[\cite{Thuillier}] \label{defn harmonic}
Let $U \subset \Xan$ be an open subset. 
We denote the space of \emph{lisse} functions in the sense of \cite[Section 3.2.1]{Thuillier} by $\AT^0(U)$. 

The space $\AT^1(U)$ is defined as the space of measures on $U$ with discrete support that is contained 
in the set of type II and III points of $\Xan$.

We denote by $\HS_X(U)$ the kernel of the operator $\ddc \colon \AT^0(U) \to \AT^1(U)$ constructed in \cite[Th\'eor\`eme 3.2.10]{Thuillier}.
The elements in $\HS_X(U)$ are called \emph{harmonic} functions on $U$. 
\end{defn}

\begin{bem}
Roughly speaking, a function is lisse if it can be locally 
written as the pullback along the retraction of a piecewise linear 
function on some skeleton. 
It is shown in \cite{Thuillier} that both $U \mapsto \AT^0(U)$ and $U \mapsto \AT^1(U)$ define sheaves 
in the Berkovich topology of $\Xan$ and 
that $\ddc$ is a morphism of sheaves, which implies that $U \mapsto \HS_X(U)$ also defines a sheaf on $\Xan$. 
We consider the sheaves $L^0$ and $L^1$ as sheaves in the Berkovich topology. 
Thuillier denotes these by $\AS^0$ and $\AS^1$, as opposed to the corresponding sheaves in the $G$-topology, 
which he denotes by $A^0$ and $A^1$.
To avoid confusion with sheaves of forms $\AS^{p,q}$, we differ from Thuillier's notation here. 

Also note that this is not the definition of $\HS_X$ that Thuillier gives, but is equivalent by \cite[Corollaire 3.2.11]{Thuillier}.
\end{bem}

Recall that a sheaf $\FS$ on a topological space $X$ is soft if  for every closed subset $B$ of $X$ the map 
$\FS(X) \to \varinjlim \limits_{U \supset B} \FS(U)$, where $U$ are open subsets of $X$ containing $B$, 
is surjective. 
\begin{lem} \label{lem fine sheaves}
The sheaf $\AT^0$ is soft.
\end{lem}

\begin{proof}
Let $B \subset \Xan$ be a closed subset.
Take $f \in \AT^0(U)$ for an open subset $U$ containing $B$.
Take $U_1$ to be another neighborhood of $B$ that furthermore satisfies $\overline{U}_1 \subset U$. 
Let $U_2 := U \setminus B$ and $\US := \{ U_1,U_2\}$. 
Using the identical arguments as in the proof of \cite[Lemme 3.3.5]{Thuillier} 
we find a locally finite refinement $\VS$ of $\US$ and for all $V \in \VS$ a lisse function $f_V \in L^0(U)$ with compact support, 
such that $\supp(f_V) \subset \overline{V}$ and $\sum_{V \in \VS} f_V = f$. 
Now taking $g = \sum \limits_{V \in \VS, V \subset U_1 } f_V$ gives a lisse function on $U_1$ that has 
support away from the boundary of $U_1$.
Thus $g$ may be extended by zero to a lisse function on all of $X$.
Further it agrees with $f$ on a neighborhood of $B$, as was required.  
\end{proof}

\subsection{Tropical Dolbeault cohomology}

Chambert-Loir and Ducros introduced a double complex $(\AS^{\bullet, \bullet}, d', d'')$ of sheaves of smooth 
real differential forms on Berkovich analytic spaces \cite{CLD}. 
If $X$ is a curve over a non-archimedean field, then we have $\AS^{p,q} = 0$ unless $0 \leq p \leq 1$ and $0 \leq q \leq 1$. 
Elements of $\AS^{0,0}$ are functions.  
We call these functions \emph{smooth} and write $C^\infty := \AS^{0,0}$. 

\begin{defn}
For all $p,q$ we denote $\AS^{p,q,\cl} := \ker(d'' \colon \AS^{p,q} \to \AS^{p,q+1})$. 

Furthermore we denote $\Aff := \ker(d'd'' \colon C^\infty \to \AS^{1,1})$ and 
call this the sheaf of \emph{affine functions}. 

We also define the \emph{tropical Dolbeault cohomology} of $X$ to be
\begin{align*}
\HH^{p,q}(\Xan) := \frac {\AS^{p,q,\cl}(\Xan)} {d''(\AS^{p,q-1}(\Xan))}. 
\end{align*}
\end{defn}

\begin{bem} \label{bem singular cohomology}
We have $\HH^{p,q}(\Xan) \cong \HH^q(\Xan, \AS^{p,0,\cl})$ by \cite[Corollary 3.4.6]{JellThesis} and 
$\AS^{0,0,\cl} \cong \underline{\R}$ by \cite[Lemma 3.4.5]{JellThesis}. 
This implies $\HH^{0,q}(\Xan) \cong \HH^{q}_{\sing}(\Xan, \R)$, the $q$-th singular cohomology
of $\Xan$ with values in $\R$. 
\end{bem}

The double complex $\AS^{\bullet,\bullet}$ admits a wedge-product and 
if $X$ is of dimension one, canonical integration of $(1,1)$-forms with compact support. 
A version of Stokes' theorem ensures that the induced pairing descends to cohomology \cite[Theorem 5.17]{Gubler}. 
In particular, we obtain for a projective curve $X$ pairings
\begin{align} \label{PD pairings}
&\HH^{0,0}(\Xan) \times \HH^{1,1}(\Xan) \to \R \text{ and} \\
\nonumber &\HH^{1,0}(\Xan) \times \HH^{0,1}(\Xan) \to \R.
\end{align}

\begin{defn} \label{defn PD}
We say that \emph{Poincar\'e duality holds for $X$} if the pairings in (\ref{PD pairings}) are perfect pairings. 
\end{defn}

\section{Harmonic functions, affine functions and $(1,1)$-tropical Dolbeault cohomology}

In this section we prove the existence of certain exact sequences relating the cohomology of the sheaves 
$\KS$ and $\HS_X$ to tropical Dolbeault cohomology.

\subsection{Sheaf cohomology of $\HS_X$}

In this section $X$ is a smooth curve over $K$.

\begin{lem}  \label{lem surjective sheaf}
Let $X$ be a smooth curve. 
The sequence
\begin{align} \label{seq resolution}
0 \to \HS_X \to \AT^0 \overset{\ddc}{\to} \AT^1 \to 0 
\end{align}
is exact.
\end{lem}
\begin{proof}
Exactness at $\HS_X$ and $\AT^0$ follows from the definition. 
We check exactness at $\AT^1$ on stalks. 
Take $x \in \Xan$. 
If $x$ is of type I or IV, then $\AT^1_x$ is zero, so we are done. 

If $x$ is of type II or III, the real vector space $\AT^1_x$ is spanned by $\delta_x$, the Dirac measure at $x$. 
Now take any other type II or III point $y \in \Xan$. 
Denote by $\overline{X}$ the canonical smooth compactification of $X$. 
By \cite[Proposition 3.3.7]{Thuillier} there exists a function $g \in \AT^0({\overline{X}}^{\an})$ such that 
$\ddc g = \delta_x - \delta_y$. 
Now the stalk $g_x$ of $g|_{\Xan}$ at $x$ satisfies $\ddc(g_x) = \delta_x$.
\end{proof}

\begin{prop} \label{prop harmonic cohomology}
Let $X$ be a smooth projective curve. 
Then there are canonical isomorphisms $\HH^0(\Xan, \HS_X) \cong \R$ and $\HH^1(\Xan, \HS_X) \cong \R$.
\end{prop}
\begin{proof}
Since $\Xan$ is a paracompact and Hausdorff topological space, soft sheaves on $\Xan$ are acyclic.
Thus we have $\HH^1(\Xan, \AT^0) = 0$ by Lemma \ref{lem fine sheaves}. 
Looking at the long exact sequence in cohomology associated to (\ref{seq resolution}) we now find that
\begin{align*}
0 \to \HH^0(\Xan, \HS_X) \to \AT^0(\Xan) \to \AT^1(\Xan) \to \HH^1(\Xan, \HS_X) \to 0
\end{align*}
is exact. 
Thus, we have to calculate the kernel and cokernel of 
$\ddc \colon \AT^0(\Xan) \to \AT^1(\Xan)$. 
The kernel is $\HS_X(\Xan)$, which consists of constant functions by the maximum principle \cite[Proposition 3.1.1]{Thuillier}. 

The isomorphism of the cokernel with $\R$ is given by total mass
\begin{align*}
\int \colon \AT^1(\Xan) \to \R; \;\; \mu \mapsto \int 1 \; \mu.
\end{align*}
For $f \in \AT^0(\Xan)$ we have $\int 1 \ddc f = \int f \ddc 1 = 0$
by \cite[Proposition 3.2.12]{Thuillier}. 
Now let $\mu = \sum_{i = 0}^{k} \lambda_i \delta_{x_i} \in \AT^1(\Xan)$ be a discrete measure of mass $0$. 
Then by \cite[Proposition 3.3.7]{Thuillier} there exist functions $g_i := g_{x_ix_0} \in \AT^0(\Xan)$ such that $\ddc (g_i ) = \delta_{x_i} - \delta_{x_0}$.
Write $g := \sum_{i=1}^k \lambda_i g_i$. 
We obtain $\ddc(g) = \sum_{i = 1}^k \lambda_i \delta_{x_i} + (\sum_{i = 1}^{k} - \lambda_i) \delta_{x_0}$. 
Since $\mu$ was of mass zero, we have $\lambda_0 = \sum_{i = 1}^{k} - \lambda_i$ which shows  $\mu = \ddc g$. 
Thus integration induces an isomorphism $\AT^1(\Xan) / \ddc \AT^0(\Xan)$ with $\R$, which is precisely what we wanted to show.
\end{proof}

\subsection{Relation between $\HS_X$ and $\Aff$}

In this section $X$ is again a smooth curve over $K$.

\begin{defn} \label{defn G(X)}
We write $G(X)$ for the set of type II points $x$ of $\Xan$ where the genus of the $\Ktilde$-curve $C_x$ is positive. 
We define the sheaf
\begin{align} \label{eq: defn SX}
\Scal_X := \bigoplus_{x \in G(X)} \iota_{x, *} (\Pic^0(C_x) \otimes_\Z \R),
\end{align}
where $\iota_x \colon \{x\} \to \Xan$ denotes the inclusion, 
and the real vector space
\begin{align} \label{eg: defn SXII}
S_X := \bigoplus_{x \in G(X)}  (\Pic^0(C_x) \otimes_\Z \R) = \Scal_X(\Xan).
\end{align}
\end{defn}

\begin{bem}
The set $G(X)$ is finite \cite[Remark 4.18]{BPR2}. 
We could let the sums in (\ref{eq: defn SX}) and (\ref{eg: defn SXII}) run over all type II points on $\Xan$, 
since if $x \notin G(X)$ we have $C_x \cong \mathbb{P}^1_{\widetilde{K}}$ and consequently $\Pic^0(C_x) = 0$. 
\end{bem}

\begin{defn}
A model $\XS$ of $X$ over $K^\circ$ is called $\emph{semistable}$ if the special fiber
$\XS_s := \XS \times_{K^\circ} \Ktilde$ has only ordinary double points as singularities. 
It is called \emph{strictly semistable} if all irreducible components of $\XS_s$ are smooth. 
\end{defn}

\begin{bem} \label{remark semistable}
Let $X$ be a smooth projective curve over $K$
and $\XS$ a semistable model of $X$ over $K^\circ$. 
There exists a canonical reduction map $\red \colon \Xan \to \XS_s$. 

Let $C$ be an irreducible component of $\XS_s$ with generic point $\eta$. 
Then $\eta$ has a unique preimage under $\red$ \cite[Theorem 4.6]{BPR2}. 
This preimage, denoted by $x$, is a type II point such that $C_x \cong C'$, where $C'$ is the normalization of $C$. 
For all points of type II that do not appear in this way we have $C_x \cong \mathbb{P}^1_{\Ktilde}$ 
\cite[Remarks 4.17 \& 4.18]{BPR2}.

In other words
\begin{align} \label{eq semistable}
S_X = \bigoplus_{\text{irrd. comp. } C \text{ of } \XS_s} \Pic^0(C') \otimes_\Z \R
\end{align}
for every semistable model $\XS$ of $X$. 
\end{bem}

\begin{bem} \label{H to S}
Thuillier defines a map $\HS_X \to \Scal_X$ as follows \cite[Proof of Lemma 2.3.22]{Thuillier}:
The tangent directions at a type II point $x \in \Xan$ are in canonical bijection 
with $C_x(\Ktilde)$. 
Thus for a harmonic function $h$ defined in a neighborhood of $x$ there is 
an associated $\R$-divisor $\divisor_x(h) := \sum d_{x, v_p}(h) \cdot p \in \Div(C_x) \otimes \R$, 
where $v_p$ denotes the tangent direction associated to $p$ and $d_{x, v_{p}}(h)$ 
denotes the slope of $h$ in the direction of $v_p$. 
Since $h$ is harmonic at $x$, we have $\divisor_x(h) \in \Div^0(C_x) \otimes \R$.
We thus obtain a map $\HS_{X, x} \to \Pic^0(C_x) \otimes \R = \Scal_{X,x}$ on stalks
by mapping $h$ to the class of $\divisor_x(h)$ in $\Pic^0(C_x) \otimes \R$. 
These maps glue to define a map $\HS_X \to \Scal_X$.  
\end{bem}

\begin{prop} \label{thm harmonic sequence}
Let $X$ be a smooth curve. 
The sequence of sheaves 
\begin{align} \label{seq harmonic}
0 \to \Aff \to \HS_X \to \Scal_X \to 0
\end{align}
is exact, where the map $\HS_X \to \Scal_X$ is the one constructed in Remark \ref{H to S}.
\end{prop}
\begin{proof}
It is shown in \cite[Lemme 2.3.22]{Thuillier} that 
\begin{align*}
0 \to \R \log \vert \OS_X^\times \vert \to \HS_X \to \Scal_X \to 0
\end{align*}
is exact. 
Here $\R \log \vert \OS_X^\times \vert$ denotes the sheaf of real-valued functions that can be locally 
written as $\sum_{i =1}^{k} \lambda_i \cdot \log \vert f_i \vert$ for $\lambda_i \in \R$ and $f_i$ invertible functions. 
It is then shown in \cite[Theorem 5.2.4]{Wanner} that $\Aff \cong \R \log \vert \OS_X^\times \vert$. 
\end{proof}

\begin{kor} \label{kor harmonic sequence}
Let $X$ be a smooth projective curve. 
The sequence 
\begin{align*}
0 \to S_X  \to \HH^1(\Xan, \Aff) \to \HH^1(\Xan, \HS_X) \to 0
\end{align*}
of real vector spaces is exact. 
In particular, we have $ \dim \HH^1(\Xan, \Aff) = \dim S_X + 1$.
\end{kor}
\begin{proof}
We consider the long exact cohomology sequence associated to (\ref{seq harmonic}):
\begin{align*}
0 \to &\HH^0(\Xan, \Aff) \to \HH^0(\Xan, \HS_X) \to \HH^0(\Xan, \Scal_X) \to \\
&\HH^1(\Xan, \Aff) \to \HH^1(\Xan, \HS_X) \to \HH^1(\Xan, \Scal_X) \to \dots
\end{align*}
Now we have $\HH^0(\Xan, \Aff) \cong \R$ by \cite[Theorem 3.2.61]{JellThesis} and 
$\HH^0(\Xan, \HS_X) \cong \R$ by Proposition \ref{prop harmonic cohomology}. 
This implies that the map $\HH^0(\Xan, \HS_X) \to \HH^0(\Xan, \Scal_X)$ is zero. 
Further we have $\HH^0(\Xan, \Scal_X) = S_X$, and $\HH^1(\Xan, \Scal_X) = 0$ since $\Scal$ is a finite sum 
of skyscraper sheaves. 
This yields the claimed short exact sequence.

The last statement follows since $\HH^1(\Xan, \HS_X) \cong \R$ by Proposition \ref{prop harmonic cohomology}.
\end{proof}

\subsection{Relation between $\HH^{1,1}(\Xan)$ and $\HH^1(\Xan, \Aff)$}

Recall that we defined $\AS^{1,0, \cl} := \ker(d'' \colon \AS^{1,0} \to \AS^{1,1})$. 

\begin{prop} \label{thm exponential sequence}
The sequence 
\begin{align} \label{exponential sequence}
0 \to \underline{\R} \to \Aff  \overset{d'}{\to} \AS^{1,0,\cl} \to 0
\end{align}
of sheaves on $\Xan$ is exact. 
\end{prop}

Let us mention here that the tropical analogue of this sequence was introduced by Mikhalkin and Zharkov \cite[Definition 4.1]{MikZharII}.
Note here that the functions in $\KS$ are precisely the ones that locally factor as an affine function through a tropicalization.

\begin{proof}
Note that $d'$ indeed maps $\Aff$ to $\AS^{1,0,\cl}$, since if $f \in \Aff$ we have
$d'' d' f = -d' d'' f = 0$. 
We have $\ker ( d' \colon C^\infty \to \AS^{1,0}) = \underline{\R}$ by \cite[Corollary 4.6]{Jell}.
Since $\underline{\R}$ is a subsheaf of $\Aff$ we obtain exactness at $\Aff$. 

We check exactness at $\AS^{1,0,\cl}$ on stalks. 
Let $x \in \Xan$ and $\alpha \in \AS^{1,0,\cl}(U)$ for an open subset $U$ of $\Xan$ containing $x$. 
Then $d' \alpha \in \AS^{2,0}(U) = 0$, thus by \cite[Theorem 4.5]{Jell} there exists 
an open subset $V \subset U$ containing $x$ and $f \in C^{\infty}(V)$
such that $d' f = \alpha$. 
Since $d' d'' f = - d'' \alpha = 0$ we have $f \in \Aff(V)$. 
This shows exactness on the stalk of $x$ at $\AS^{1,0,\cl}$. 
\end{proof}

\begin{kor} \label{kor exponential sequence}
For a smooth projective curve $X$ we obtain the exact sequence 
\begin{align*}
0 \to \HH^{1,0}(\Xan) \to \HH^{0,1}(\Xan) \to \HH^1(\Xan, \Aff) \to \HH^{1,1}(\Xan) \to 0
\end{align*}
of real vector spaces. 
\end{kor}
\begin{proof}
The long exact sequence in cohomology associated to (\ref{exponential sequence}) reads
\begin{align*}
0 \to &\HH^0(\Xan, \underline{\R}) \to \HH^0(\Xan, \Aff) \to \HH^0(\Xan, \AS^{1,0,\cl}) \\
\to &\HH^1(\Xan, \underline{\R}) \to \HH^1(\Xan, \Aff) \to \HH^1(\Xan, \AS^{1,0,\cl}) \to \HH^2(\Xan, \underline{\R}) \to \dots.
\end{align*}
Now using $\HH^0(\Xan, \Aff) \cong \R$ \cite[Theorem 3.2.61]{JellThesis}, 
$\HH^q(\Xan, \AS^{p,0,\cl}) \cong \HH^{p,q}(\Xan)$ \cite[Corollary 3.4.6]{JellThesis},
$\AS^{0,0, \cl} \cong \underline{\R}$ \cite[Lemma 3.4.5]{Jell}, and  
the fact that $\Xan$ is homotopy equivalent to a finite one-dimensional simplicial complex, 
which implies $\HH^2(\Xan, \underline{\R}) = 0$, we obtain
\begin{align*}
0 \to \R \to \R \to \HH^{1,0}(\Xan) \to \HH^{0,1}(\Xan) \to \HH^1(\Xan, \Aff) \to \HH^{1,1}(\Xan) \to 0.
\end{align*}
Now by exactness the map $\R \to \HH^{1,0}(\Xan)$ is zero and we obtain the result. 
\end{proof}

\section{Poincar\'e duality and finite dimensionality}

\subsection{Poincar\'e duality and finite dimensionality for smooth projective curves}

In this section, $X$ is a smooth projective curve over $K$. 

We write $h^{p,q}(\Xan) := \dim_\R \HH^{p,q}(\Xan)$ and call these numbers \emph{tropical Hodge numbers}. 

\begin{lem} \label{lem non degenerate left}
The Poincar\'e duality pairings 
\begin{align*}
&\HH^{0,0}(\Xan) \times \HH^{1,1}(\Xan) \to \R \text{ and} \\
&\HH^{1,0}(\Xan) \times \HH^{0,1}(\Xan) \to \R 
\end{align*}
are non-degenerate on the left, 
i.e.~for all $\alpha \in \HH^{p,0}(\Xan) \setminus \{0\}$ there exists $\beta \in \HH^{1-p,1}(\Xan)$ 
such that $\langle \alpha, \beta \rangle \neq 0$. 

In particular, Poincar\'e duality holds for $X$ if and only if $h^{1,1}(\Xan) = 1$ and $h^{1,0}(\Xan) = h^{0,1}(\Xan)$. 
\end{lem}
\begin{proof}
We have that $\HH^{0,0}(\Xan)$ is the space of constant functions. 
Since the integration map $\AS^{1,1}(\Xan) \to \R$ is non-trivial, $\R$-linear and descends to cohomology, 
the first pairing is non-degenerate on the left.

By \cite[Proposition 3.4.11]{JellThesis} we have an injective map 
$J \colon \HH^{1,0}(\Xan) \to \HH^{0,1}(\Xan)$ with the property that
$\int \alpha \vedge J\alpha \neq 0$ if $\alpha \neq 0$.
This shows that the second pairing is non-degenerate on the left. 

The last statement is a direct consequence of the fact that $h^{1,0}(\Xan)$ is finite and $h^{0,0}(\Xan) = 1$. 
\end{proof}

\begin{satz} \label{PD}
For a smooth projective curve $X$, Poincar\'e duality holds if and only if $S_X = 0$. 
\end{satz}

\begin{proof}
If Poincar\'e duality holds, the injective map $\HH^{1,0}(\Xan) \to \HH^{0,1}(\Xan)$ from Corollary \ref{kor exponential sequence} is an isomorphism.
Thus $\HH^1(\Xan, \Aff) \to \HH^{1,1}(\Xan)$ is also an isomorphism. 
Since the latter space has dimension $1$ again by Poincar\'e duality, so does $\HH^1(\Xan, \Aff)$. 
By Corollary \ref{kor harmonic sequence} we have $\dim S_X = \dim \HH^1(\Xan, \Aff) - 1$ and hence $S_X = 0$. 
 
Conversely if $S_X = 0$, then again by Corollary \ref{kor harmonic sequence} we have $\HH^1(\Xan, \Aff) \cong \R$. 
Since $h^{1,1}(\Xan) \geq 1$, the surjective map $\HH^1(\Xan, \Aff) \to \HH^{1,1}(\Xan)$ 
from Corollary \ref{kor exponential sequence} is an isomorphism. 
Thus the injective map $\HH^{1,0}(\Xan) \to \HH^{0,1}(\Xan)$ from Corollary \ref{kor exponential sequence} is also an isomorphism, 
which shows that we have Poincar\'e duality by Lemma \ref{lem non degenerate left}. 
\end{proof}

\begin{kor} \label{corollary finite field}
Let $K$ be an algebraically closed non-archimedean field whose residue field is the algebraic closure of a finite field. 
Then Poincar\'e duality holds for smooth projective curves $X$ over $K$. 
\end{kor}
\begin{proof}
We have to show that for any point $x \in \Xan$ of type II with associated $\Ktilde$-curve $C_x$ 
the group $\Pic^0(C_x)$ is torsion. 
Now $\Pic^0(C_x)$ is the set of $\Ktilde$-points of an abelian variety $J_x$, namely the Jacobian of $C_x$. 
Any element of $J_x(\Ktilde)$ is defined over some finite subfield of $\Ktilde$, thus is torsion. 
\end{proof}

Recall the subset $G(X)$ of $\Xan$, defined in \ref{defn G(X)}, which is the finite set of type II points of positive genus. 

\begin{satz} \label{non finite dimensional}
Let $X$ be a smooth projective curve over $K$ and let the residue field of $K$ be $\C$. 
Then $\HH^{1,1}(X)$ is finite dimensional if and only if $G(X) = \emptyset$. 
\end{satz}
\begin{proof}
If $G(X) = \emptyset$, then $S_X = 0$ and thus $X$ satisfies Poincar\'e duality by Theorem \ref{ThmA} 
(or equivalently by \cite[Theorem 2]{JellWanner}). 
Since $\HH^{0,0}(\Xan) \cong \R$, we then have $\HH^{1,1}(\Xan) \cong \R$, thus $\HH^{1,1}(\Xan)$ is in particular 
finite dimensional. 

We now show that $G(X) \neq \emptyset$ implies that $\HH^{1,1}(\Xan)$ is not finite dimensional. 
The vector space $\HH^{0,1}(\Xan)$ is finite dimensional by \cite[Theorem 4.9]{Jell}. 
By Corollary \ref{kor exponential sequence} we thus have to show that $\HH^1(\Xan, \Aff)$ is not finite dimensional if $G(X) \neq \emptyset$. 
By Corollary \ref{kor harmonic sequence} it is thus enough to show that $S_X$ is not finite dimensional when $G(X) \neq \emptyset$. 
If $x \in G(X)$, then $C_x$ is a smooth projective curve of positive genus $g$ over $\C$. 
Thus $\Pic^0(C_x) \cong \C^g / \Gamma$ for some lattice $\Gamma$. 
In particular $\Pic^0(C_x)$ is not of finite rank as an abelian group, thus $\dim_\R (\Pic^0(C_x) \otimes_\Z \R)$ is not finite 
(since the tensor product is taken over $\Z$). 
Therefore $\dim_\R S_X$ is not finite.
\end{proof}

\begin{bem}
We want to remark that curves with $G(X) \neq \emptyset$ actually exist. 
One way to see this is that any curve over $K$ with $G(X) = \emptyset$ is either $\mathbb{P}^1_K$ or a Mumford curve 
\cite[Proposition 2.26 \& Theorem 2.28]{JellWanner}. 
More concretely, by Remark \ref{remark semistable}, any curve of positive genus with good reduction has non-empty $G(X)$. 
\end{bem}

\subsection{From global Poincar\'e duality to local Poincar\'e duality} \label{sect. global to local}

In this section $X$ is a smooth curve over a non-archimedean field $K$.

We now change our perspective on Poincar\'e duality by considering forms with compact support. 
For any open subset $U \subset \Xan$ we define $\AS^{p,q}_c(U)$ to be the vector subspace
of $\AS^{p,q}(U)$ of forms with compact support.
We define 
\begin{align*}
\HHc^{p,q}(U) := \frac {\ker (d'' \colon \AS^{p,q}_c(U) \to \AS^{p,q+1}_c(U))}{\im (d'' \colon \AS^{p,q-1}_c(U) \to \AS^{p,q}_c(U))}.
\end{align*} 
Integration and the wedge product induce a pairing
\begin{align*}
\AS^{p,q}(U) \times \AS^{1-p,1-q}_c(U) \to \R.
\end{align*}
This descends to a pairing on cohomology by a suitable version of Stokes' theorem \cite[Theorem 5.17]{Gubler}.
We will from now on consider the induced map
\begin{align*}
\PD_U \colon \HH^{p,q}(U) \to \HHc^{1-p,1-q}(U)^* := \Hom_\R(\HH^{1-p,1-q}_c(U), \R).
\end{align*}
We say that $U$ \emph{satisfies Poincar\'e duality} ($\PD$) if $\PD_U$ is an isomorphism for all $p,q$.  
If $\Xan$ is compact, this agrees with Definition \ref{defn PD} since $\HH^{p,q}(\Xan)  = \HHc^{p,q}(\Xan)$ in that case.
Note that $U$ satisfies $\PD$ if it is disjoint from $G(X)$ by \cite[Theorem 2]{JellWanner}. We will crucially use this fact.

\begin{lem} \label{lem 3 out of 4}
Let $U$ be an open subset of $\Xan$.
\begin{enumerate}
\item
Let $U = U_1 \cup U_2$ and
write $U_{12} = U_1 \cap U_2$. 
Then if three of the four sets $U, U_1, U_2, U_{12}$ satisfy $\PD$, so does the fourth. 
\item
Let $U = \bigcup_{i=1}^{k} U_k$ be a \emph{disjoint} union of open subsets and 
suppose that $U$ satisfies $\PD$. 
Then for every subset $J \subset \{1,\dots,k\}$ the subset $\bigcup_{j \in J} U_j$ satisfies $\PD$.
\end{enumerate}
\end{lem}
\begin{proof}
To prove $i)$ the standard arguments from differential geometry work in our setting. 
See for instance the book by Bott and Tu on differential forms \cite[Lemma 5.6]{BottTu}.
Since the sheaves $\AS^{p,q}$ admit partitions of unity \cite[Proposition 3.3.6]{CLD} both $\HH^{p,q}$ and $\HHc^{p,q}$ satisfy 
Mayer-Vietoris-sequences \cite[Appendix A.1]{JellThesis}. 
Dualizing the Mayer-Vietoris-sequence for $\HHc^{p,q}$, we see that $\PD$ induces a morphism of complexes 
between the Mayer-Vietoris-sequence for $\HH^{p,q}$ and the one for ${\HHc^{1-p,1-q}}^*$.
The claim now follows from the five lemma.

For $ii)$, notice that we have $\PD_U = \bigoplus_{i = 1} ^k \PD_{U_i}$. 
Thus if $\PD_U$ is an isomorphism, so is $\PD_{\bigcup_{j \in J}  U_j } = \bigoplus_{j \in J} \PD_{U_j}$.  
\end{proof}

\begin{satz} \label{PD local}
Let $X$ be a smooth curve and $U \subset \Xan$ an open subset. 
Assume further that $S_X = 0$. 
Then $U$ satisfies $\PD$.
\end{satz}
\begin{proof}
Replacing $X$ by its canonical smooth compactification, we may assume that $X$ is projective. 
Then we still have $S_X = 0$. 

We first treat the case where $U$ contains $G(X)$. 
Recall that $G(X)$ is the finite set of type II points such that $C_x$ has positive genus.
Now $\Xan$ satisfies $\PD$ by Theorem \ref{PD}
and $\Xan \setminus G(X)$ and $U \setminus G(X)$ do so by \cite[Theorem 2]{JellWanner}, since they do not contain
any type II points of positive genus. 
Then using Lemma \ref{lem 3 out of 4} $i)$ with $U = \Xan$, $U_1 = U$ and $U_2 = \Xan \setminus G(X)$ 
shows that $U$ satisfies $\PD$.

Now we treat the general case. 
We write $G(X) = \{x_1,\dots,x_k\}$. 
We choose for each $i$ an open neighborhood $V_i$ of $x_i$ such that 
$(V_i)_{i =1,\dots,k}$ are pairwise disjoint. 
By intersecting with $U$ if necessary, we may assume that if $x_i \in U$, we have $V_i \subset U$. 
Let $J$ be the set of $i$ such that $x_i \in U$. 
Since $\bigcup_{i = 1} ^k V_i$ contains $G(X)$, it satisfies $\PD$ by our reduction above. 
Now by Lemma \ref{lem 3 out of 4} $ii)$ we know that $\bigcup_{j \in J} V_j$ 
also satisfies $\PD$. 
We write $ U_1 = U \setminus G(X)$ and  $U_2 = \bigcup_{j \in J} V_j$. 
Both $U_1$ and $U_{12}$ satisfy $\PD$ since they do not contain any type II points of positive genus \cite[Theorem 2]{JellWanner}. 
Since we just showed that $U_2$ satisfies $\PD$, so does $U$ by Lemma \ref{lem 3 out of 4} $i)$. 
\end{proof}

\subsection{Dimensions of cohomology and a conjecture by Y. Liu}

For curves $X$ that satisfy $S_X = 0$, Poincar\'e duality enables us to completely calculate 
the tropical Hodge numbers $h^{p,q}(U) := \dim_\R \HH^{p,q}(U)$ 
both for $U = \Xan$ as well as for a basis of the topology.   
The calculation of $h^{p,q}(\Xan)$ enables us to establish the case of curves for a conjecture by Y. Liu. 

For the case of a basis of the topology, we follow closely the proofs of Wanner and the author in \cite[Section 5]{JellWanner},
using our stronger Poincar\'e duality results from Theorem \ref{PD local}.  

\begin{satz}\label{thm dimensions global}
Let $X$ be a smooth projective curve over $K$ such that $S_X = 0$.
We denote by $b$ the dimension of the first singular cohomology $\HH^1_{\sing}(\Xan, \R)$ of $\Xan$. 
Then we have
\begin{align*}
h^{p,q}(\Xan) = 
\begin{cases} 
1 \text{ if }(p,q) = (0,0),(1,1), \\
b \text{ if }(p,q) = (0,1),(1,0).
\end{cases}
\end{align*}
\end{satz}
Note that $b$ is also the first Betti number of the incidence graph of a semistable model of $X$ or
equivalently of any skeleton of $\Xan$ \cite[4.9]{BPR2}.
\begin{proof}
We have $h^{0,0}(\Xan) = 1$  and $h^{0,1}(\Xan) = b$ by \cite[Corollary 4.6]{Jell}. 
Thus  $h^{1,1}(\Xan) = 1$ and $h^{1,0}(\Xan) = b$ follow from Theorem \ref{PD}.
\end{proof}

For a variety $X$ over $K$ and $p \geq q$, 
Y. Liu constructed maps  $N \colon \HH^{p,q}(\Xan) \to \HH^{q,p}(\Xan)$ and conjectures 
that these are isomorphisms when $K$ is an algebraically closed non-archimedean field whose
residue field is the algebraic closure of a finite field and $X$ is smooth and proper \cite[Conjecture 5.2]{Liu2}. 
He established this for $(p,q) = (1,0)$ under some additional assumptions \cite[Theorem 5.12]{Liu2}.
We show that when $X$ is a curve, these additional assumptions are not necessary. 

\begin{satz} \label{proof Liu conjecture}
Let $X$ be a smooth projective curve over $K$, where the residue field of $K$ is the algebraic closure of a finite field. 
Then Liu's conjecture holds for $X$. 
\end{satz}
\begin{proof}
The only case we have to consider is $(p,q) = (1,0)$. 
Then the map $N$ coincides with the operator $J$ \cite[Lemma 3.2 (1)]{Liu2}, which we know to be injective by \cite[3.4.11]{JellThesis}. 
We have $S_X = 0$ since the residue field of $K$ is the algebraic closure of a finite field. 
Thus by Theorem \ref{thm dimensions global}, the dimensions of $\HH^{1,0}(\Xan)$ and $\HH^{0,1}(\Xan)$ agree and are finite, 
which shows that the injective map $N$ is an isomorphism. 
\end{proof}

As in \cite{JellWanner}, for local considerations we use results on the structure of non-archimedean curves from \cite{BPR2}, 
as well as the notion of semistable vertex sets defined there. 
A semistable vertex set is a finite subset $V$ of $\Xan$ such that $\Xan \setminus V$ is isomorphic 
to a disjoint union of open balls and finitely many open annuli. 
The following definition is inspired by \cite[Corollary 4.27 \& Definition 4.28]{BPR2} and also appeared in \cite[Definition 5.2]{JellWanner}.

\begin{defn}\label{Def. simple}
Let $X$ be a smooth curve. 
An open subset $U$ of $\Xan$ is called \emph{simple} if it is either isomorphic to an open disc
or an open annulus or if there exists a semistable vertex set $V$, with associated skeleton $\Sigma(X, V)$ and retraction $\tau_V$,
and a simply connected open subset $W$ of $\Sigma(X, V)$
 such that 
$U= \tau_{V}^{-1}(W)$. 
A simple open subset $U$ is called \emph{strictly simple}
if its closure in $\Xan$ is simply connected in the case of a disc or an annulus or 
if the closure of $W$ in $\Sigma(X, V)$ is simply connected for $U=\tau^{-1}_V(W)$.
\end{defn}

\begin{bem} \label{rem properties}
Strictly simple subsets form a basis of the topology  
and have finitely many boundary points \cite[Proposition 5.3 \& Corollary 5.5]{JellWanner}. 
If $U$ is a strictly simple open subset with $k$ boundary points, it is properly homotopy equivalent 
to the one point union of $k$ half open intervals, glued together at the closed ends \cite[Lemma 5.6]{JellWanner}.
\end{bem}

\begin{satz} \label{Theorem Teilmenge Koho.}
Let $X$ be a smooth projective curve over $K$ such that $S_X = 0$. 
Let $U$ be a strictly simple open subset of $\Xan$ and denote by $k := \# \del U$ the finite number of boundary points. 
Then we have 
	\begin{align*}
	h^{p,q}(U) = 
	\begin{cases}
	1 &\text{ if } (p,q) = (0,0) \\
	k - 1 &\text{ if } (p,q) = (1,0) \\
	0 &\text{ else } 
	\end{cases}
	\text{ and }
	h^{p,q}_c(U) = 
	\begin{cases}
	1 &\text{ if } (p,q) = (1,1) \\
	k - 1 &\text{ if } (p,q) = (0,1) \\
	0 &\text{ else}.
	\end{cases}
	\end{align*}
\end{satz}

\begin{proof}
First, note that $U$ has $\PD$ by Theorem \ref{PD local}.
Since further $h^{p,q}(U)$ and $h^{p,q}_c(U)$ are zero unless $(p,q) \in \{0, 1\}^2$, 
it is sufficient to calculate $h^{0,q}(U)$ and $h^{0,q}_c(U)$ for $q = 0, 1$. 
Since $\HH^{0,q}(U) \cong \HH^{q}_{\sing}(U, \R)$ , 
we only have to calculate the Betti numbers $h^{q}_{\sing}(U)$ and $h^{q}_{c, \sing}(U)$ of singular cohomology,
which are invariant under proper homotopy equivalences. 
Thus by Remark \ref{rem properties}, we have to calculate $h^{q}_{\sing}(Y)$ and $h^{q}_{c, \sing}(Y)$ for $Y$  a one point union of $k$ intervals. 
The details are carried out in the proof of \cite[Theorem 5.7]{JellWanner}.
\end{proof}

\bibliographystyle{alpha}
\def\cprime{$'$}

\end{document}